\documentclass[a4paper,12pt]{article}
\usepackage{amsmath,amssymb,amsthm}
\usepackage{fullpage,hyperref,enumitem}

\setlist{nosep} 

\theoremstyle{plain}
\newtheorem{theorem}{Theorem}[section]
\newtheorem{lemma}[theorem]{Lemma}
\newtheorem{proposition}[theorem]{Proposition}
\newtheorem{corollary}[theorem]{Corollary}

\theoremstyle{remark}
\newtheorem{remark}[theorem]{Remark}

\theoremstyle{definition}
\newtheorem{definition}[theorem]{Definition}

\setlist{nosep} 

\DeclareMathOperator{\diam}{diam}

\newcommand{\bE}{{\mathbb{E}}}
\newcommand{\bN}{{\mathbb{N}}}
\newcommand{\bP}{{\mathbb{P}}}
\newcommand{\bR}{{\mathbb{R}}}
\newcommand{\bZ}{{\mathbb{Z}}}
\newcommand{\cA}{{\mathcal{A}}}
\newcommand{\hrho}{{\hat{\rho}}}
\newcommand{\eps}{{\varepsilon}}

\title{Equidistribution for nonuniformly expanding dynamical systems,
and application to the almost sure invariance principle}

\author{Alexey Korepanov
\\ {\small Mathematics Institute, University of Warwick,
Coventry, CV4 7AL, UK}}

\date{January 13, 2017 (updated October 21, 2017)}

\begin{document}

\maketitle

\begin{abstract}
  Let $T \colon M \to M$ be a nonuniformly expanding dynamical system,
  such as logistic or intermittent map. 
  Let $v \colon M \to \mathbb{R}^d$
  be an observable and $v_n = \sum_{k=0}^{n-1} v \circ T^k$ denote the Birkhoff sums.
  Given a probability measure $\mu$ on $M$, we consider $v_n$
  as a discrete time random process on the probability space $(M, \mu)$.
  
  In smooth ergodic theory there are various natural choices of $\mu$,
  such as the Lebesgue measure, or the absolutely continuous $T$-invariant
  measure. They give rise to different random processes.
  
  We investigate relation between such processes.
  We show that in a large class of measures, it is possible to couple
  (redefine on a new probability space)
  every two processes so that they are almost surely close to each other,
  with explicit estimates of ``closeness''.
  
  The purpose of this work is to close a gap in the proof of the almost sure
  invariance principle for nonuniformly hyperbolic transformations
  by Melbourne and Nicol.
\end{abstract}

\section{Introduction}
\label{sec:intro}

Suppose that \(T \colon M \to M\) is a dynamical system,
and \(v \colon M  \to \bR^d\) is an observable. 
Let \(v_n = \sum_{k=0}^{n-1} v \circ T^k\) denote the Birkhoff sums. 
Given a probability measure \(\mu\) on \(M\), let
\((v_n, \mu)\) denote the discrete time random process given by
\(v_n\) on the probability space \((M,\mu)\).

In the study of statistical properties of \(v_n\), such as the
the central limit theorem, various choices for \(\mu\) come up naturally,
giving rise to different random processes.

For example, if \(M=[0,1]\) and \(T\) is a nonuniformly
expanding map as in Young~\cite{Y99} such as intermittent or logistic
with a Collet-Eckmann parameter, 
then \(\mu\) may be
(a) the Lebesgue measure,
(b) the absolutely continuous invariant probability measure (a.c.i.p.),
(c) the a.c.i.p.\ for the associated \emph{induced map}
(see Section~\ref{sec:UEM}).

The interest in the Lebesgue measure comes from physics:
it is a \emph{natural choice} of initial condition.
The a.c.i.p.\ has an important advantage over the Lebesgue measure:
if \(\mu\) is the a.c.i.p., then the increments of the process \((v_n, \mu)\) are stationary.
It is standard to prove and state limit theorems in terms of the a.c.i.p.

The measure in (c) appears in a widely used technical argument,
when \(T\) is reduced by a time change \emph{(inducing)} to a uniformly expanding map,
which may be easier to work with. Then statistical properties
of the induced map are used to prove results on the original map.

We explore the relation between processes defined with respect to
different measures. Our motivation is the study of almost sure approximations
by Brownian motion.

\begin{definition}
  \label{defn:ASIP}
  We say that \(v_n\) satisfies the 
  \emph{Almost Sure Invariance Principle} (ASIP), if without changing the distribution,
  \(\{v_n, n \geq 0\}\) can be redefined on a new probability space
  with a Brownian motion \(W_t\), such that with some \(\beta < 1/2\),
  \[
    v_n = W_n + o(n^{\beta})
    \qquad \text{almost surely}
    .
  \]
\end{definition}

The ASIP is a strong statistical property,
it implies the central limit theorem (CLT) and
the law of iterated logarithm (LIL), which in one dimension take form
\[
  \bP\Bigl( \frac{v_n}{\sqrt{n}} \in [a,b] \Bigr)
  \xrightarrow{n \to \infty} \frac{1}{\sqrt{2 \pi \sigma^2} }\int_a^b e^{-\frac{x^2}{2 \sigma^2}} \, dx
  \quad \text{for all } a \leq b
\]
and
\[
  \limsup_{n \to \infty} \frac{v_n}{\sqrt{n \log \log n}} = \sqrt{2} \sigma
  \quad \text{almost surely}.
\]
The ASIP also implies functional versions of the CLT and the LIL as well
as other laws, see Philipp and Stout \cite[Chapter~1]{PS75}.

Melbourne and Nicol \cite{MN05,MN09} proved

\begin{theorem}
  Suppose that \(T\) is nonuniformly expanding with return times in 
  \(L^p\), \(p > 2\) (see Section~\ref{sec:UEM} for definitions)
  with an absolutely continuous invariant probability measure \(\rho\). 
  If \(v \colon M \to \bR^d\) is a H\"older continuous continuous
  observable with \(\int_M v \, d\rho = 0\), then the process
  \(v_n = \sum_{k=0}^{n-1} v \circ T^k\), defined on a probability space
  \((M, \rho)\), satisfies the ASIP.
\end{theorem}

\begin{remark}
  \label{rmk:NUH}
  Following the approach of \cite{B75,S72},
  the ASIP for nonuniformly expanding systems extends to
  a large class of nonuniformly hyperbolic systems 
  which satisfy the hypotheses of Young~\cite{Y98},
  for example Sinai billiards or H\'enon maps. See \cite[Lemma~3.2]{MN05}.
\end{remark}

Later Gou\"ezel discovered a gap in~\cite{MN05,MN09}:
what Melbourne and Nicol actually proved is the ASIP 
for a different starting measure, the one invariant invariant under the
induced map. A similar issue appears in Denker and Philipp~\cite{DP84},
though they do not claim the ASIP for the invariant measure.
Even though there is a close relation between the two measures,
the argument relating the ASIP-s was missing. The main goal of this
paper is to fill this gap.

\begin{remark}
  Despite the gap, the usual corollaries of the
  ASIP (such as the functional central limit theorem
  and functional law of iterated logarithm) can be obtained from~\cite{MN05,MN09},
  as it is done in~\cite{DP84}.
\end{remark}

\begin{remark}
  Besides~\cite{MN05,MN09}, there are other results which
  cover nonuniformly hyperbolic systems, but only partially:
  \begin{itemize}
    \item Chernov~\cite{C06}: scalar ASIP for dispersing billiards.
    \item Gou\"ezel~\cite{G10}: vector valued ASIP for dynamical systems
      with an exponential multiple decorrelation assumption (includes
      dispersing billiards).
    \item Cuny and Merlev\`ede~\cite{CM15}: scalar ASIP for
      reverse martingale differences (applies to nonuniformly expanding maps,
      see~\cite{KKM16mart}).
  \end{itemize}
  Problems which are only covered by~\cite{MN05} and~\cite{MN09}
  include the vector valued ASIP for maps with slower than
  exponential rate of decay of correlations, such as the intermittent
  family~\cite{LSV99}.
\end{remark}

We work in the setting where \(T\) is a nonuniformly 
expanding map (as in~\cite{Y99}) and \(v_n = \sum_{k=0}^{n-1} v \circ T^k\)
are Birkhoff sums. Given two probability measures \(\mu\) and \(\rho\),
we compare the random processes \(X_n = (v_n, \mu)\) and \(Y_n = (v_n, \rho)\).

Our main result is that if \(v\) is bounded, then
in a large class of probability measures, it is possible to redefine
\(\{X_n, n \geq 0\}\) and \(\{Y_n, n \geq 0\}\) on a new
probability space so that
\(
  Z = \sup_{n \geq 0} |X_n - Y_n|
\)
is finite almost surely.

\begin{remark}
  Technically, the statement above means that there exists
  a probability space \((\Omega, \bP)\), supporting processes \(X'_n\), \(Y'_n\),
  such that:
  \begin{itemize}
    \item \(\{X_n, n \geq 0\}\) is equal in distribution to \(\{X'_n, n \geq 0\}\),
    \item \(\{Y_n, n \geq 0\}\) is equal in distribution to \(\{Y'_n, n \geq 0\}\),
    \item \(Z' = \sup_{n \geq 0} |X'_n - Y'_n|\) is finite almost surely.
  \end{itemize}
\end{remark}

In addition, we estimate the tails of \(Z\) (i.e.\ \(\bP(Z \geq a)\)
for \(a \geq 0\)) in terms of \(|v|_\infty\) and parameters of \(T\)
such as distortion bound and asymptotics of return times.

For a fixed \(n \geq 0\), we estimate the distance
between \(X_n\) and \(Y_n\) in L\'evy-Prokhorov and Wasserstein
metrics. We expect such estimates to be useful for families
of dynamical systems as in~\cite{KKM15}.

\begin{remark}
  Our approach is in many ways similar to the \emph{Coupling Lemma} for
  dispersing billiards \cite[Lemma~7.24]{CM06},
  due to Chernov, Dolgopyat and Young.
  Also, after the first version of this paper was circulated,
  the author was made aware that some of the techniques are analogous to
  those in Zweim\"uller~\cite{Z09}. Notably, our disintegration~\eqref{eq:jnngg}
  corresponds to Zweim\"uller's \emph{regenerative partition of unity}.
\end{remark}

The paper is organized as follows.
In Section~\ref{sec:UEM} we give the definition of nonuniformly expanding maps
and state our results.
In Section~\ref{sec:app} we present some applications, including the
ASIP in Subsection~\ref{sec:asip}.
Section~\ref{sec:proofs} contains the proofs.

\section{Abstract setup and results}
\label{sec:UEM}

\subsection{Nonuniformly expanding maps}

We use notation \(\bN=\{1,2,\ldots\}\) and \(\bN_0 =\bN \cup \{0\}\).

Let \((M,d)\) be a metric space with a Borel probability measure \(m\)
and \(T \colon M \to M\) be a nonsingular transformation.
We assume that there exists \(Y \subset M\) with \(m(Y)>0\) and
\(\diam Y < \infty\), an at most countable partition \(\alpha\) of \(Y\)
(modulo a zero measure set)
and \(\tau \colon Y \to \bN\) with \(\int_Y \tau \, dm < \infty\)
such that for every \(a \in \alpha\),
\begin{itemize}
  \item \(m(a) > 0\),
  \item \(\tau\) assumes a constant value \(\tau(a)\) on \(a\),
  \item \(T^{\tau(a)} a \subset Y\).
\end{itemize}
Let \(F \colon Y \to Y\), \(F y = T^{\tau(y)} y\). We require that
there are constants \(\lambda>1\), \(\hat{K} \geq 0\) and \(\eta \in (0,1]\),
such that for each \(a \in \alpha\) and \(x,y \in a\):
\begin{itemize}
  \item \(F\) restricts to a (measure-theoretic) bijection
    from \(a\) onto \(Y\),
  \item \(d(F x, F y) \geq \lambda d(x,y)\),
  \item the inverse Jacobian \(\zeta_m = \frac{dm}{dm \circ F}\) of \(F\)
    has bounded distortion:
    \[
      \bigl| \log \zeta_m (x) - \log \zeta_m (y) \bigr|
      \leq \hat{K} d(F x, F y)^\eta
      .
    \]
\end{itemize}

We call such maps \(T\) \emph{nonuniformly expanding}.
We refer to \(F\) as \emph{induced map} and to \(\tau\) as
\emph{return time function.}
The class of nonuniformly expanding maps includes 
logistic maps at Collet-Eckmann parameters,
intermittent maps \cite{Y99} and Viana maps.

To simplify the exposition, we assume that \(\diam Y \leq 1\) and \(\eta = 1\).
The general case can be always reduced to this by 
replacing the metric \(d\) with \(d'\) given by \(d'(x,y) = c d(x,y)^\eta\),
where \(c\) is a sufficiently small constant.

It is standard that there exists a unique \(F\)-invariant
absolutely continuous probability
measure \(\mu\) on \(Y\). Let \(\zeta = \frac{d\mu}{d\mu \circ F}\).
By \cite[Propositions 2.3 and 2.5]{KKM16},
\begin{equation}
  \label{eq:antt}
  K^{-1} \leq \frac{d\mu}{dm} \leq K
  \qquad \text{and} \qquad
  \bigl| \log \zeta (x) - \log \zeta (y) \bigr|
    \leq K d(F x, F y)
\end{equation}
for all \(x,y \in a\), \(a \in \alpha\),
where \(K\) is a constant which depends continuously 
(only) on \(\lambda\) and \(\hat{K}\).
Where convenient, we view \(\mu\) as a measure on \(M\)
supported on \(Y\).

For a function \(\phi \colon Y \to \bR\) denote
\[
  |\phi|_\infty = \sup_{x \in Y} |\phi(x)|,
  \qquad
  |\phi|_d = \sup_{x \neq y \in Y} \frac{|\phi(x) - \phi(y)|}{d(x,y)}
  \qquad and \qquad \|\phi\|_d = |\phi|_\infty + |\phi|_d.
\]
For \(\phi \colon Y \to (0,\infty)\), denote
\(|\phi|_{d, \ell} = |\log \phi|_d\).

\subsection{Coupling of processes}

Fix a  constant \(R' > K \lambda / (\lambda - 1)\).

\begin{definition}
  \label{defn:reg}
  We call a probability measure \(\rho\) on \(M\)
  \emph{regular} if it is supported on \(Y\) and \(d \rho = \phi \, d\mu\),
  where \(\phi \colon Y \to [0, \infty)\)
  satisfies \(|\phi|_{d, \ell} \leq R'\).
\end{definition}

\begin{definition}
  \label{defn:freg}
  We say that a probability measure \(\rho\) on \(M\) is \emph{forward regular}, if
  it allows a disintegration
  \begin{equation}
    \label{eq:nu}
    \rho = \int_{E} \rho_z \, d\varkappa (z),
  \end{equation}
where \((E, \varkappa)\) is a probability space
and \(\{\rho_z\}\) is a measurable family of probability 
measures on \(M\), and
there exists a function \(r \colon E \to \bN_0\) such that
\(T_*^{r(z)} \rho_z\) is a regular measure for each \(z\). 
We refer to \(r\) as a \emph{jump function.}
\end{definition}

Define \(s \colon M \times M \to \bN_0 \cup \{\infty\}\),
\begin{equation}
  \label{eq:ffo}
  s(x,y) = 
  \inf \bigl\{ \max\{k, n\} \colon k, n \geq 0, \, T^k x = T^n y \bigr\}.
\end{equation}

Note that if \(s(x,y) < \infty\), then the trajectories
\(T^k x\) and \(T^k y\), \(k \geq 0\) coincide
up to a time shift and possibly different beginnings.

\begin{theorem}
  \label{thm:yeop}
  Suppose that a probability measure \(\rho\) on \(M\) is forward regular.
  Then there exists a probability measure \(\hrho\) on \(M \times M\)
  with marginals \(\rho\) and \(\mu\) on the first and second 
  components respectively such that \(s\) is finite \(\hrho\)-almost surely.
  
  In addition, with \((E, \varkappa)\) and \(r\) as in Definition~\ref{defn:freg},
  \begin{enumerate}[label=(\alph*)]
  \item (Weak polynomial moments)
  If \(\varkappa(r \geq n) \leq C_\beta n^{-\beta}\) and 
  \(\mu(\tau \geq n) \leq C_\beta n^{-\beta}\) for all \(n \geq 1\) with some 
  constants \(\beta > 1\) and \(C_\beta > 0\), then
  \[
    \hrho(s \geq n) \leq C n^{-\beta} \quad \text{for all } n > 0,
  \]
  where the constant \(C\) depends continuously (only) on
  \(\lambda\), \(K\), \(R'\), \(\beta\) and \(C_\beta\).
  \item (Strong polynomial moments)
  If \(\int r^\beta \, d\varkappa \leq C_\beta\) and 
  \(\int \tau^\beta \, d\mu \leq C_\beta\) with some 
  constants \(\beta > 1\) and \(C_\beta > 0\), then
  \[
    \int s^\beta \, d\hrho \leq C,
  \]
  where the constant \(C\) depends continuously (only) on
  \(\lambda\), \(K\), \(R'\), \(\beta\) and \(C_\beta\).
  \item (Exponential and stretched exponential moments)
  If \(\varkappa(r \geq n) \leq C_{\alpha,\gamma} e^{-\alpha n^\gamma}\) and 
  \(\mu(\tau \geq n) \leq C_{\alpha,\gamma} e^{-\alpha n^\gamma}\)
  for all \(n\)
  with some constants \(\alpha > 0\), \(\gamma \in (0,1]\), \(C_{\alpha, \gamma} > 0\),
  then  
  \[
    \hrho(s \geq n) \leq C e^{-A n^\gamma} \quad \text{for all } n > 0,
  \]
  where the constants \(C > 0\) and \(A > 0\) depend continuously (only) on
  \(\lambda\), \(K\), \(R'\), \(\alpha\), \(\gamma\) and \(C_{\alpha,\gamma}\).
  \end{enumerate}
\end{theorem}

Let \(v \colon M \to \bR^d\) be a bounded observable and
\(v_n = \sum_{k=0}^{n-1} v \circ T^k\). Denote
\(|v|_\infty = \sup_{x \in M} |v(x)|\).

\begin{remark}
  \label{rmk:aaggg}
  \(|v_n(x) - v_n(y)| \leq 2 |v|_\infty s(x,y)\)
  for all \(x,y \in M\) and \(n \geq 0\).
\end{remark}

Let \(\rho_j\), \(j=1,2\) be two forward regular probability measures
with disintegrations 
\(
  \rho_j = \int_{E_j} \rho_{j,z} \, d\varkappa_j (z)
\)
and jump functions \(r_j\).
Let \(X_n = (v_n, \rho_1)\) and \(Y_n = (v_n, \rho_2)\) be the
related random processes.

\begin{theorem}
  \label{thm:yeoc}
  The processes \(\{X_n, n \geq 0\}\) and \(\{Y_n, n \geq 0\}\)
  can be redefined on the same probability space \((\Omega, \bP)\)
  such that
  \(
    Z = \sup_{n \geq 0} |X_n - Y_n|
  \)
  is finite with probability one. Also:
  
  \begin{enumerate}[label=(\alph*)]
  \item (Weak polynomial moments)
  If \(\varkappa_1(r_1\geq n) \leq C_\beta n^{-\beta}\),
  \(\varkappa_2(r_2\geq n) \leq C_\beta n^{-\beta}\) and
  \(\mu(\tau \geq n) \leq C_\beta n^{-\beta}\) for all \(n\) with
  some constants \(C_\beta > 0\) and \(\beta > 1\), then
  \[
    \bP(Z \geq x) \leq C x^{-\beta} \quad \text{for all } x > 0,
  \]
  where the constant \(C\) depends continuously (only) on
  \(\lambda\), \(K\), \(R'\), \(\beta\), \(C_\beta\) and \(|v|_\infty\).
  \item (Strong polynomial moments)
  If \(\int r_1^\beta \, d\varkappa_1 \leq C_\beta\),
  \(\int r_2^\beta \, d\varkappa_2 \leq C_\beta\) and
  \(\int \tau^\beta \, d\mu \leq C_\beta\) with
  some constants \(C_\beta > 0\) and \(\beta > 1\), then
  \[
    \int Z^\beta \, d\bP \leq C,
  \]
  where the constant \(C\) depends continuously (only) on
  \(\lambda\), \(K\), \(R'\), \(\beta\), \(C_\beta\) and \(|v|_\infty\).
  \item (Exponential and stretched exponential moments)
  If \(\varkappa_1(r_1 \geq n) \leq C_{\alpha,\gamma} e^{-\alpha n^\gamma}\), 
  \(\varkappa_2(r_2 \geq n) \leq C_{\alpha,\gamma} e^{-\alpha n^\gamma}\) and 
  \(\mu(\tau \geq n) \leq C_{\alpha,\gamma} e^{-\alpha n^\gamma}\)
  for all \(n\)
  with some constants \(\alpha > 0\), \(\gamma \in (0,1]\), \(C_{\alpha, \gamma} > 0\),
  then  
  \[
    \bP(Z \geq x) \leq C e^{-A x^\gamma} \quad \text{for all } x > 0,
  \]
  where the constants \(C > 0\) and \(A > 0\) depend continuously (only) on
  \(\lambda\), \(K\), \(R'\), \(\alpha\), \(\gamma\), \(C_{\alpha,\gamma}\)
  and \(|v|_\infty\).
  \end{enumerate}
\end{theorem}

Proofs of Theorems~\ref{thm:yeop} and~\ref{thm:yeoc} are
in Section~\ref{sec:proofs}.

\section{Applications}
\label{sec:app}

\subsection{L\'evy-Prokhorov and Wasserstein distances}

Let \(X\) and \(Y\) be \(\bR^d\)-valued random variables, and \(\bP_X\), \(\bP_Y\)
be the associated probability measures on \(\bR^d\). Recall the following
definitions:

\begin{definition}
  The L\'evy-Prokhorov distance between \(X\) and \(Y\) is
  \begin{align*}
    d_{LP} (X, Y)
    = \inf \{ & \eps > 0 \colon \bP_X(A) \leq \bP_Y (A^\eps) + \eps
    \;\text{ and }\; \bP_Y(A) \leq \bP_X (A^\eps) + \eps
    \\ & \text{ for all Borel } A \subset \bR^d\},
  \end{align*}
  where \(A^\eps = \{x \colon \inf_{y \in A } |x-y| \leq \eps \} \).
\end{definition}

\begin{definition}
  For \(p \geq 1\), the \(p^{\text{th}}\) Wasserstein distance between \(X\) and \(Y\) is
  \[
    d_{W,p} (X, Y)
    = \inf \Bigl[ \bE \bigl( |X-Y|^p \bigr) \Bigr]^{1/p},
  \]
  where the infimum is taken over all couplings of \(X\) and \(Y\).
\end{definition}

Suppose that \(X_n\) and \(Y_n\) are as in Theorem~\ref{thm:yeoc} (a),
under the assumption of the polynomial tails.
Then Theorem~\ref{thm:yeoc} implies the following:

\begin{corollary}
  For each \(n \geq 0\) and \(1 \leq p < \beta\),
  \[
    d_{LP} (X_n, Y_n) \leq C_{LP}
    \qquad \text{and} \qquad
    d_{W,p} (X_n, Y_n) \leq C_{W,p}
    ,
  \]
  where the constants \(C_{LP}\) and \(C_{W,p}\) depend continuously (only) on
  \(p\) and the constant \(C\) from Theorem~\ref{thm:yeoc} (a).
  In particular, they do not depend on \(n\).
\end{corollary}

\begin{proof}
  Let \(n\) be fixed. Theorem~\ref{thm:yeoc} provides us with a 
  coupling of \(X_n\) and \(Y_n\) on a probability space
  \((\Omega, \bP)\) such that \(Z = |X_n - Y_n|\) satisfies
  \(\bP(Z \geq x) \leq C x^{-\beta}\) for all \(x > 0\).
  
  By definition, \(d_{W, p}(X_n, Y_n) \leq \bigl( \bE ( Z^p ) \bigr)^{1/p}\),
  and the bound on \(d_{W, p}(X_n, Y_n)\) follows.
  By \cite[Theorem 2]{GS02}, \(d_{LP} (X_n, Y_n)
  \leq \sqrt{d_{W,1} (X_n, Y_n)}\).
\end{proof}

\begin{remark}
  Our estimates on the distances between \(X_n\) and \(Y_n\) do not depend on
  \(n\). It follows that the distances between their normalized versions,
  such as \(n^{-1/2} X_n\) and \(n^{-1/2} Y_n\),
  converge to zero as \(n\) goes to infinity. 
\end{remark}

\subsection{Disintegration for the \(T\)-invariant measure}
\label{sec:jnu}

Recall that \(\mu\) is the absolutely continuous \(F\)-invariant measure.
Following \cite{Y99},
there exists a unique \(T\)-invariant ergodic probability measure \(\rho\)
on \(M\), with respect to which \(\mu\) is absolutely continuous.

To define the regular measures, we fix
\(R' > K \lambda / (\lambda - 1)\).
Here we show that \(\rho\) fits the setup of 
Theorems~\ref{thm:yeop} and~\ref{thm:yeoc}:

\begin{proposition}
  \label{prop:jnu}
  The measure \(\rho\) is forward regular:
  \(
    \rho = \int_{E} \rho_z \, d\varkappa (z),
  \)
  with jump function \(r \colon E \to \bN_0\) such that
  \(\varkappa(r = n) = \bar{\tau}^{-1} \mu(\tau \geq n)\),
  where \( \bar{\tau} = \int_Y \tau \, d\mu\).
\end{proposition}

\begin{proof}
  We start by constructing a Young tower
  \(
    \breve{M} 
    = \{(y, \ell) \in Y \times \bZ \colon 0 \leq \ell < \tau(y) \} 
  \)
  with the tower map
  \[
    \breve{T} (y, \ell) = 
    \begin{cases}
      (y, \ell+1), & \ell < \tau(y) - 1, \\
      (Fy, 0), & \ell = \tau(y) - 1
    \end{cases}
    .
  \]
  The projection \(\pi \colon \breve{M} \to M\), \(\pi(y,\ell) = T^\ell(y)\)
  serves as a semiconjugacy between \(\breve{T}\) and \(T\).
  The natural probability measure
  \(
    \breve{\rho} = 
    \mu \times \text{counting} 
  \)
  on \(\breve{M}\) is \(\breve{T}\)-invariant,
  and its projection \(\rho = \pi_* \breve{\rho}\) is the only
  \(T\)-invariant ergodic probability measure \(M\) 
  such that \(\mu \ll \rho\).

  Using the definition of \(\breve{\rho}\) and 
  \(\pi\), we can write \(\rho\) as
  \[
    \rho
    = \bar{\tau}^{-1}
    \sum_{a \in \alpha} \sum_{\ell=0}^{\tau(a)-1}
    \mu(a) T_*^{\ell} \mu_a
    ,
  \]
  where \(\mu_a\) is the normalized restriction of \(\mu\) to \(a\), i.e.\ 
  \(\mu_a (S) = (\mu(a))^{-1} \mu (a \cap S)\)
  for all \(S \subset M\).
  
  Let \(E = \{ (a, \ell) \in \alpha \times \bZ 
  \colon 0 \leq \ell < \tau(a)\}\) and 
  \(\varkappa(a, \ell) = \bar{\tau}^{-1} \mu(a)\).
  Then \(\varkappa\) is a probability measure on \(E\), and
  \[
    \rho = \sum_{(a, \ell) \in E} \rho_{a,\ell} \, \varkappa(a, \ell),
    \qquad \text{where} \,\,
    \rho_{a, \ell} = T_*^\ell \mu_a
  \]
  is the disintegration we are after.
  Further, let \(r \colon E \to \bZ\), \(r(a, \ell) = \tau(a)-\ell\).
  Then for every \(a, \ell\), 
  the measure \(T_*^{r(a, \ell)} \rho_{a, \ell} = F_* \mu_a\)
  is supported on \(Y\), and its density
  is \(p_a(y) = (\mu(a))^{-1} \zeta({y_a})\), 
  where \(y_a\) is the unique preimage of \(y\)
  in \(a\) under \(F\). 
  By \eqref{eq:antt}, \(T_*^{r(a, \ell)} \rho_{a, \ell}\) is regular.

  Finally,
  \begin{align*}
    \varkappa(r = n) 
    & = \sum_{(a,\ell) \in E} 1_{\ell = \tau(a) -  n} \varkappa(a,\ell)
    = \bar{\tau}^{-1} \sum_{a \in \alpha \colon \tau(a) \geq n} \mu(a)
    = \bar{\tau}^{-1} \mu(\tau \geq n)
    .
  \end{align*}

\end{proof}

\subsection{Intermittent maps}
\label{sec:LSV}

Consider a family of Pomeau-Manneville maps, as in \cite{LSV99},
\(T \colon [0,1] \to [0,1]\),
\[
  T (x) = 
  \begin{cases}
    x ( 1 + 2^\gamma x^\gamma), & x \leq 1/2 \\
    2 x - 1,                    & x > 1/2
  \end{cases}
  ,
\]
where \(\gamma \in (0,1)\) is a parameter. This is a popular example
of maps with polynomial decay of correlations (sharp rate for
H\"older observables is \(n^{1-1/\gamma}\) \cite{G04,H04,S02,Y99}).

Let \(M = [0,1]\). It is standard (see \cite{Y99}) that 
\(T\) fits the setup of Section~\ref{sec:UEM}
with \(Y = [1/2,1]\), and \(\tau\) being
the first return time to \(Y\).

We consider three natural probability measures on \(M\):
\begin{itemize}
  \item \(m\), the Lebesgue measure,
  \item \(\rho\), the unique absolutely continuous measure,
  \item \(\mu\), the absolutely continuous invariant measure
    for the induced map, as in Section~\ref{sec:UEM}.
\end{itemize}

Let \(v \colon M \to \bR^d\) be a bounded observable, 
\(v_n = \sum_{k=0}^{n-1} v \circ T^k\), and
\(X_{m,n} = (v_n, m)\), \(X_{\rho,n} = (v_n, \rho)\)
and \(X_{\mu,n} = (v_n, \mu)\) be the corresponding
random processes.

\begin{theorem}
  \label{thm:sshw}
  The processes \(\{X_{m,n}, n \geq 0\}\),
      \(\{X_{\mu,n}, n \geq 0\}\) and
      \(\{X_{\rho,n}, n \geq 0\}\) can be redefined on the
      same probability space \((\Omega, \bP)\) so that
  \begin{itemize}
    \item 
      \(Z_{m,\mu} = \sup_{n \geq 0} |X_{m,n} - X_{\mu,n}|\)
      satisfies \(\bP (Z_{m, \mu} \geq x) \leq C x^{-1/\gamma}\)
      for \(x > 0\).
    \item 
      \(Z_{m,\rho} = \sup_{n \geq 0} |X_{m,n} - X_{\rho,n}|\)
      satisfies \(\bP (Z_{m, \rho} \geq x) \leq C x^{-1/\gamma + 1}\)
      for \(x > 0\).
    \item 
      \(Z_{\rho,\mu} = \sup_{n \geq 0} |X_{\rho,n} - X_{\mu,n}|\)
      satisfies \(\bP (Z_{\rho, \mu} \geq x) \leq C x^{-1/\gamma + 1}\)
      for \(x > 0\).
  \end{itemize}
  The constant \(C\) depends continuously (only) on \(\gamma\) and \(|v|_\infty\).
\end{theorem}

\begin{proof}
  We write \(a \ll b\), if there is a constant \(C\) which 
  depends continuously only on \(\gamma\) such that \(a \leq C b\).

  It is enough to show that with an appropriate choice of
  the constant \(R'\) in Definition~\ref{defn:reg},
  the measures \(m\) and \(\rho\) are forward regular:
  \begin{itemize}
    \item[(a)]
      \(
        m = \int_{E_m} m_z \, d\varkappa_m(z)
      \)
      with \(r_m \colon E_m \to \bN_0\)
      for which \(T^{r_m(z)}_* m_z\) are regular probability measures.
      Also, \(\varkappa_m( r_m \geq n) \ll n^{-1/\gamma}\) for all \(n > 0\).

    \item[(b)]
      \(
        \rho = \int_{E_\rho} \rho_z \, d\varkappa_\rho(z)
      \)
      with \(r_\rho \colon E_\rho \to \bN_0\)
      for which \(T^{r_\rho(z)}_* \rho_z\) are regular probability measures.
      Also, \(\varkappa_\rho (r_\rho \geq n) \ll n^{-1/\gamma + 1}\)
      for all \(n > 0\).
  \end{itemize}
  Then the results follow from Theorem~\ref{thm:yeoc} and Lemma~\ref{lemma:joc}.
  
  We use the bound \(\mu(\tau \geq n) \ll n^{-1/\gamma}\), 
  (see \cite{K16} for the proof with uniform constants).
  By Proposition~\ref{prop:jnu}, \(\rho\) is forward regular and
  \[
    \varkappa_\rho(r_\rho \geq n) = \sum_{k \geq n} \varkappa_\rho(r_\rho = k)
    = \bar{\tau}^{-1} \sum_{k \geq n} \mu(\tau \geq k)
    \ll n^{-1/\gamma + 1}
    .
  \]
  This proves (b). Further we prove (a).

  We extend \(\tau \colon Y \to \bN\) to \(\tau \colon M \to \bN\)
  by \(\tau(x) = \min \{ k \geq 1 \colon T^k (x) \in Y\}\),
  and accordingly set \(F \colon M \to Y\), 
  \(F(x) = T^{\tau(x)}(x)\), extending the previous definition.

  It is standard \cite{K16} that \(M\) can be partitioned 
  (modulo a zero measure set) into countably many 
  subintervals \([a_k, b_k]\), \(k \in \bN\), on which
  \(\tau\) is constant, and \(F \colon [a_k, b_k] \to Y\)
  is a diffeomorphism with bounded distortions, i.e.\ 
  \begin{equation}
    \label{eq:ann}
    \Bigl| \log \frac{F'(x)}{F'(y)} \Bigr|
    \ll |F(x) - F(y)|
    \qquad \text{ for all } k \text{ and } x,y \in [a_k, b_k].
  \end{equation}
  Further, \(m(\tau \geq n) \ll n^{-1/\gamma}\).

  Let \(m_k\) denote the normalized Lebesgue measure on \([a_k, b_k]\).
  It follows from~\eqref{eq:ann} and~\eqref{eq:antt} that \(F_* m_k\) is a regular
  measure with \(R'\) depending continuously (only) on \(\gamma\).

  It follows that \(m\) is forward regular:
  \(
    m = \sum_{k \in \bN} \varkappa_m(k) m_k,
  \)
  with the probability space \((\bN, \varkappa_m)\),
  \(\varkappa_m(k) = |b_k - a_k|\), and \(r_m \colon \bN \to \bN_0\),
  \(r_m(k) = \tau\bigr|_{[a_k, b_k]}\).
  
  Finally, observe that \(\varkappa_m( r_m \geq n) \ll n^{-1/\gamma}\).
\end{proof}

\subsection{Almost sure invariance principle}
\label{sec:asip}

Let \(v \colon M \to \bR^d\), and \(v_n = \sum_{k=0}^{n-1} v \circ T^k\).
Recall that \(\mu\) is the absolutely continuous \(F\)-invariant measure
on \(Y\).
Let \(\rho\) be the \(T\)-invariant measure on \(M\) as in 
Subsection~\ref{sec:jnu}. Suppose that \(\int_M v \, d\rho = 0\).
Let \(X_n = (v_n, \rho)\) and \(Y_n = (v_n, \mu)\).

Under the assumptions that \(\tau \in L^p\), \(p > 2\) and 
\(v\) is H\"older continuous, Melbourne and Nicol prove in \cite{MN05,MN09} the 
ASIP for \(Y_n\) (with rates), and claim the ASIP for \(X_n\).
However, their argument does not cover the transition from \(Y_n\) to \(X_n\).
Here we close this gap.

\begin{theorem}
  \label{thm:vahu}
  The ASIP for \(X_n\) is equivalent to the ASIP for \(Y_n\), with the same rates.
\end{theorem}

\begin{remark}
  In~\cite{MN05,MN09}, the authors prove the ASIP for nonuniformly expanding systems
  and then extend the result to nonuniformly hyperbolic systems \cite[Section 3]{MN05}.
  In Theorem~\ref{thm:vahu}, \(T\) is a nonuniformly expanding system,
  but proving it, we close the gap in both situations.
\end{remark}

\begin{proof}[Proof of Theorem~\ref{thm:vahu}]
  Assume the ASIP for \(X_n\) as in Definition \ref{defn:ASIP},
  with a Brownian motion \(W_n\) and rate \(o(n^\beta)\).
  
  Proposition~\ref{prop:jnu} allows us to use Theorem~\ref{thm:yeoc} to
  redefine the processes \(\{X_n, n \geq 0\}\) and \(\{Y_n, n \geq 0\}\)
  on the same probability space so that
  \(\sup_{n \geq 0} |X_n - Y_n|\) is finite almost surely.
  
  Using Lemma~\ref{lemma:joc}, we can redefine \(\{X_n, n \geq 0\}\),
  \(\{Y_n, n \geq 0\}\) and \(W_t\)
  on the same probability space so that 
  \(\sup_{n \geq 0} |X_n - Y_n| < \infty\) and
  \(X_n = W_n + o( n^{\beta} )\) almost surely. 
  Then also \(Y_n = W_n + o( n^{\beta} )\) almost surely.
  
  We proved that the ASIP for \(X_n\) implies the ASIP for \(Y_n\),
  with the same rates. The same argument proves the other direction.
\end{proof}

\section{Proof of Theorems~\ref{thm:yeop} and~\ref{thm:yeoc}}
\label{sec:proofs}

\subsection{Outline of the proof}

Recall that \(\mu\) is the absolutely continuous probability measure, invariant
under the induced map \(F\).
To prove Theorem~\ref{thm:yeop}, we:
\begin{enumerate}[label=(\alph*)]
  \item\label{thm:yeop:rpu} Build (Subsection~\ref{sec:dis}) a countable probability space \(\cA\)
    with a function \(t \colon \cA \to \bN_0\)
    and show that if \(\rho\) is a probability measure such that \(T_*^n \rho\) is 
    regular for some \(n \geq 0\), then \(\rho\) has a representation
    \begin{equation}
      \label{eq:jnngg}
      \rho = \sum_{a \in \cA} \bP(a) \rho_k
      \qquad \text{with } T_*^{n + t(a)} \rho_a = \mu
      \text{ for all } a
      ,
    \end{equation}
    where \(\bP\) is a probability measure on \(\cA\).
    (C.f.\ \emph{regenerative partition of unity} in \cite{Z09}).
  \item Show that the tails \(\bP(t \geq n)\) can be bounded uniformly for all
    regular measures (Subsection~\ref{sec:tails}).
  \item\label{thm:yeop:diss}
    Consider a particularly simple case, when \(\rho\) is such that \(T_*^n \rho = \mu\)
    for some \(n \geq 0\). Then we take \(\hrho = (U_n)_* \rho\), where
    \(U_n \colon M \to M \times M\), \(U_n(x) = (x, T^n x)\).
    We observe that the marginals of \(\hrho\) on the first and second coordinates
    are \(\rho\) and \(\mu\) respectively and \(\hrho(s \geq n) = 0\).
  \item\label{thm:yeop:dis} 
    The procedure in \ref{thm:yeop:diss} transparently extends to weighted sums of measures,
    as in \eqref{eq:jnngg}.
    We take 
    \[
      \hrho = \sum_{a \in \cA} \bP(a) (U_{n + t(a)})_* \rho_a
      .
    \]
    Observe that then \(\hrho(s \geq n + k ) \leq \bP(t \geq k)\) for all \(k \geq 0\).
  \item Now, \ref{thm:yeop:rpu} and \ref{thm:yeop:dis} already prove 
    Theorem~\ref{thm:yeop} for the case when \(T_*^n \rho\)
    is regular. In Subsection~\ref{proof:thm:yeop} we extend this to the class of
    all forward regular measures.
\end{enumerate}

The idea of the proof of Theorem~\ref{thm:yeoc} is that
if \(\rho_1\) and \(\rho_2\) are forward regular measures,
then each of them can be coupled with
\(\mu\) in the sense of Theorem~\ref{thm:yeop}.
Then we couple \(\rho_1\) and \(\rho_2\) through their couplings with \(\mu\) by
a standard argument in Probability Theory, see Appendix~\ref{app:joc}.

\subsection{Disintegration}
\label{sec:dis}

Let $P\colon L^1(Y)\to L^1(Y)$ be the transfer operator corresponding
to $F$ and $\mu$, so 
$\int_Y P\phi\,\psi\,d\mu=\int_Y\phi\,\psi\circ F\,d\mu$
for all $\phi\in L^1$ and $\psi\in L^\infty$.
Then $P \phi$ is given explicitly by
\[
  (P \phi)(y) = \sum_{a \in \alpha} \zeta(y_a) \phi(y_a),
\]
where $y_a$ is the unique preimage of $y$ under $F$ lying in $a$.

Recall that \(R'\) is a fixed constant, and \(R' > K \lambda / (\lambda -1)\).
Let \(R = \lambda(R' - K)\). Then \(R > K + \lambda^{-1} R\).
Choose \(\xi \in (0,e^{-R})\) such that 
\(R (1-\xi e^R) \geq K + \lambda^{-1} R\).

\begin{proposition}
  \label{prop:bamh}
  Assume that \(\phi \colon Y \to (0, \infty)\) is such that
  \(|\phi|_{d, \ell} \leq R'\).
  Then \(\phi = \xi \int_Y \phi \, d\mu + \psi\), where
  \(|\psi|_{d, \ell} \leq R\).
  In addition, \(|P(1_a \psi)|_{d, \ell} \leq R'\)
  for every \(a \in \alpha\).
\end{proposition}

\begin{proof}
  See \cite[Propositions 3.1 and 3.2]{KKM16}.
\end{proof}

Let \(\cA\) denote the countable set of all finite words in the alphabet
\(\alpha\), including the empty word. For \(a \in \cA\), let 
\([a]\) denote the subset of words in \(\cA\)
which begin with \(a\). Let \(\ell(a)\) denote the length of \(a\).
Define \(t \colon \cA \to \bZ\), 
\(t(a)=\sum_{k=1}^{\ell(a)} \tau(a_k)\), where \(a_k\) is the \(k\)-th
letter of \(a\).

\begin{proposition}
  \label{prop:vaoe}
  Let \(\rho\) be a probability measure on \(M\) such that
  \(T_*^n \rho\) is regular for some \(n \geq 0\). Then there is a
  decomposition
  \(\rho = \xi \rho' + \sum_{a \in \alpha} r_a \rho_a\),
  where \(\rho'\) and all \(\rho_a\) are probability measures
  and \(r_a > 0\), such that
  \begin{itemize}
    \item
      \(
        e^{-R} (1-\xi) \mu(a)
        \leq r_a \leq
        e^{ R} (1-\xi) \mu(a)
      \),
    \item \(T_*^n \rho' = \mu\),
    \item \(T_*^{n+\tau(a)} \rho_a\) is a regular measure
      for every \(a \in \alpha\).
  \end{itemize}
\end{proposition}

\begin{proof}
  Let \(\chi = T_*^n \rho\). Since \(\chi\) is regular probability measure,
  there exists \(\phi \colon Y \to (0, \infty)\) such that
  \(|\phi|_{d, \ell} \leq R'\),
  \(d\chi = \phi \, d\mu\) and \(\int_Y \phi \, d\mu = 1\).
  
  By Proposition~\ref{prop:bamh},
  \( \phi = \xi + \psi\), where \(|\psi|_{d, \ell} \leq R\).
  For \(a \in \alpha\), define \(r_a = \int_a \psi \, d\mu\) 
  and \(\psi_a = r_a^{-1} 1_a \psi \).
  Then \(\int_Y \psi_a \, d\mu = 1\) and
  by Proposition~\ref{prop:bamh}, \(|P \psi_a|_{d, \ell} \leq R'\).
  Define \(\chi_a\) to be a probability 
  measure on \(M\) given by 
  \( d \chi_a = \psi_a \, d\mu\).
  Then \(T_*^{\tau(a)} \chi_a\) is
  a regular probability measure with density \(P \psi_a\). 
  
  Observe that
  \begin{equation}
    \label{eq:hhja}
    \chi = \xi \mu + \sum_{a \in \alpha} r_a \chi_a.
  \end{equation}
  
  By \cite[(3.1)]{KKM16},
  \[
    e^{-R} (1-\xi) =
    e^{-R} \int_Y \psi \, d\mu 
    \leq \psi \leq 
    e^{ R} \int_Y \psi \, d\mu 
    = e^R (1-\xi).
  \]
  Therefore 
  \(
    e^{-R} (1-\xi) \mu(a) 
    \leq r_a \leq 
    e^{ R} (1-\xi) \mu(a)
  \).
  
  Now we use~\eqref{eq:hhja} to decompose \(\rho\) similarly.
  Define \(\rho'\) to be a measure on \(M\) given by
  \(
    \frac{d\rho'}{d\rho} = \frac{d\mu}{d\chi} \circ T^n.
  \)
  Then \(T_*^n \rho' = \mu\).
  Similarly define \(\rho_a\), \(a \in \alpha\) by
  \(
    \frac{d\rho_a}{d\rho} = \frac{d\chi_a}{d\chi} \circ T^n.
  \)
  Then \(T_*^n \rho_a = \chi_a\). Finally note that
  \(\rho = \xi \rho' + \sum_{a \in \alpha} r_a \rho_a\).
\end{proof}

\begin{lemma}
  \label{lemma:b87e}
  Let \(n \geq 0\) and \(\rho\) be a probability measure 
  on \(M\) such that \(T_*^n \rho\) is regular.
  There exists a probability measure \(\bP\) on \(\cA\)
  and a disintegration
  \begin{equation}
    \label{eq:atyi}
    \rho = \sum_{a \in \cA} \bP(a) \rho_a,
  \end{equation}
  where \(\rho_a\), \(a \in \alpha\) are probability measures on \(M\)
  such that \(T_*^{n+t(a)} \rho_a = \mu\).
  The measure \(\bP\) satisfies 
  \begin{equation}
    \label{eq:ajiw}
    \begin{gathered}
      \bP(\ell = k) = (1-\xi)^k \xi,
      \\
      e^{-R} (1-\xi) \mu(a_{k+1}) 
      \leq \, \bP([a_1 \cdots a_{k+1}] \mid [a_1 \cdots a_k]) 
      \leq e^{ R} (1-\xi) \mu(a_{k+1}),
    \end{gathered}
  \end{equation}
  for all \(k \geq 0\) and \(a_1, \ldots, a_{k+1} \in \alpha\).
\end{lemma}

\begin{proof}
  Write \(\rho = \xi \rho' + \sum_{x \in \alpha} r_x \rho_x\) as in
  Proposition~\ref{prop:vaoe}. Then for each \(x \in \alpha\) apply
  Proposition~\ref{prop:vaoe} again and write
  \(\rho_x = \xi \rho_x' + \sum_{y \in \alpha} r_{xy} \rho_{xy} \).
  Apply the same to each \(\rho_{xy}\) and so on.
  Then
  \[
    \rho = \xi \rho' + \sum_{x \in \alpha} r_x \xi \rho_x'
    + \sum_{x,y \in \alpha} r_x r_{xy} \xi \rho_{xy}' + \cdots
  \]
  This is a disintegration as in \eqref{eq:atyi} with
  \(\bP(a) = r_{a_1} r_{a_1 a_2} \cdots r_{a_1 a_2 \cdots a_n} \xi\)
  for \(a = a_1 \cdots a_n \in \cA\).
  Conditions~\eqref{eq:ajiw} are immediate.
\end{proof}

\subsection{Polynomial and exponential tails}
\label{sec:tails}

Let \(\rho\) be a measure as in Lemma~\ref{lemma:b87e}
and \(\bP\) be the corresponding measure on \(\cA\).
Recall that \(t \colon \cA \to \bZ\) is the word length.

In this subsection we obtain elementary estimates of moments of \(t\)
in situations when \(\int_Y \tau^p \, d\mu < \infty\)
for some \(p > 1\), or
\(\int_Y e^{\gamma \tau} \, d\mu < \infty\) for some 
\(\gamma > 0\).

For \(n \geq 1\), let \(\cA_n\) be the subset of \(\cA\) of 
all words of length \(n\). 
By Lemma~\ref{lemma:b87e}, \(\bP(\cA_n) = (1-\xi)^n \xi\).
Let \(\bP_n\) denote the conditional probability measure on \(\cA_n\).

Elements of \(\cA_n\) have the form \(a = a_1 \cdots a_n\),
and \(a_1, \ldots, a_n\) can be considered as random variables
with values in \(\alpha\), and \(t = \tau(a_1) + \cdots + \tau(a_n)\).

It follows from Lemma~\ref{lemma:b87e} that
for all \(k \leq n\) and \(x \in \alpha\),
\begin{equation}
  \label{eq:lgg}
  \bP_n(a_k = x \mid a_1, \ldots, a_{k-1}) 
  \leq e^R \mu(x)
\end{equation}

\subsubsection{Polynomial tails}

\begin{proposition}
  \label{prop:fwqjq}
  Suppose that there exist \(C_\tau > 0\) and \(\beta > 1\) such that
  \(m(\tau \geq \ell) \leq C_\tau \ell^{-\beta}\) for \(\ell \geq 1\).
  Then \(\bP(t \geq \ell) \leq C \ell^{-\beta}\), where the constant
  \(C > 0\) depends continuously on \(R\), \(\xi\) and \(C_\tau\).
\end{proposition}

\begin{proof}
  Let \(k \leq n\), and \(a = a_1 \cdots a_n \in \cA_n\). 
  By \eqref{eq:lgg},
  \[
    \bP_n(\tau(a_k) \geq \ell) 
    \leq e^R m(\tau \geq \ell)
    \leq C_\tau e^R \ell^{-\beta}
    .
  \]
  Next,
  \[
    \bP_n( t \geq \ell) 
    \leq \sum_{k=1}^n \bP_n(\tau(a_k) \geq \ell/n)
    \leq n C_\tau e^R (\ell/n)^{-\beta}
    .
  \]
  Finally,
  \[
    \bP(t \geq \ell)
    = \sum_{n=1}^\infty \bP(\cA_n) \bP_n(t \geq \ell)
    \leq C_\tau e^R \xi \ell^{-\beta} \sum_{n=1}^\infty (1-\xi)^n n^{1+\beta}
    .
  \]
\end{proof}

\begin{proposition}
  \label{prop:qjqwf}
  Suppose that there exist \(C_\tau > 0\) and \(\beta > 1\) such that
  \(\int \tau^\beta \, dm \leq C_\tau\).
  Then \(\int t^\beta \, d\bP \leq C\), where the constant
  \(C > 0\) depends continuously on \(R\), \(\xi\) and \(C_\tau\).
\end{proposition}

\begin{proof}
  Let \(k \leq n\), and \(a = a_1 \cdots a_n \in \cA_n\). 
  By \eqref{eq:lgg},
  \[
    \int \tau^\beta(a_k) \, d\bP_n
    \leq e^R \int \tau^\beta \, dm
    \leq C_\tau e^R
    .
  \]
  Next,
  \[
    t^\beta (a) 
    = (\tau(a_1) + \cdots + \tau(a_n))^\beta
    \leq n^{\beta-1} (\tau^\beta(a_1) + \cdots + \tau^\beta(a_n)),
  \]
  thus
  \[
    \int t^\beta \, d\bP_n
    \leq n^{\beta-1} C_\tau e^R
    .
  \]
  Finally,
  \[
    \int t^\beta \, d\bP
    = \sum_{n=1}^{\infty} \bP(\cA_n)  \int t^\beta \, d\bP_n
    \leq C_\tau e^R \xi \sum_{n=1}^\infty (1-\xi)^n n^{\beta-1} 
    .
  \]
\end{proof}

\subsubsection{(Stretched) exponential tails}

\begin{proposition}
  \label{prop:aexp}
  Let \(X_1, \ldots, X_n\) be nonnegative random variables.
  Suppose that there exist \(\alpha>0\), $\gamma\in(0,1]$, such that
  \[
    \bP( X_k \geq \ell\,|\,X_1=x_1,\dots,X_{k-1}=x_{k-1}) 
    \leq C e^{-\alpha \ell^\gamma}
  \]
  for all \(\ell \geq 0\), $1\le k\le n$ and $x_1,\dots,x_{k-1}\ge0$.
  Then for all $A \in (0,\alpha/2]$, $\ell\ge0$,
  \[
    \bP(X_1 + \cdots + X_n \geq \ell)
    \leq (1+A C_1)^n e^{-A \ell^\gamma},
  \]
  where \(C_1\) depends continuously 
  on $C$, \(\gamma\) and \(\alpha\).
\end{proposition}

\begin{proof}
  See \cite[Proposition 4.11]{KKM16}.
\end{proof}

\begin{proposition}
  \label{prop:afsame}
  Suppose that there exist \(C_\tau > 0\), \(\alpha > 0\) and
  \(\gamma \in (0,1]\) such that
  \(m(\tau \geq \ell) \leq C_\tau e^{-\alpha \ell^{\gamma}}\)
  for \(\ell \geq 1\).
  Then \(\bP(t \geq \ell) \leq C e^{ - A \ell^{\gamma}}\),
  where the constants \(C > 0\) and \(A \in (0, \alpha)\)
  depend continuously on \(R\), \(\xi\), \(C_\tau\), \(\alpha\)
  and \(\gamma\).
\end{proposition}

\begin{proof}
  Let \(k \leq n\), and \(a = a_1 \cdots a_n \in \cA_n\).
  By \eqref{eq:lgg},
  \[
    \bP_n(\tau(a_k) \geq \ell \mid a_1, \ldots, a_{k-1})
    \leq e^R m(\tau \geq \ell)
    \leq C_\tau e^{-\alpha \ell^\gamma}
    .
  \]
  By Proposition~\ref{prop:aexp},
  \[
    \bP_n(t \geq \ell)
    \leq (1+A C_1)^n e^{-A \ell^\gamma}
  \]
  for all \(A \in (0, \alpha/2)\).
  Taking \(A\) small enough, we obtain  
  \[
    \bP(t \geq \ell)
    = \sum_{n=1}^{\infty} \bP(\cA_n) \bP_n(t \geq \ell)
    \leq \xi e^{-A \ell^\gamma} \sum_{n=1}^{\infty} (1-\xi)^n (1+A C_1)^n
    = C e^{-A \ell^\gamma}
  \]
  with \(C < \infty\).
\end{proof}

\subsection{Coupling}

Recall that \(s \colon M \times M \to \bN_0 \cup \{\infty\}\) is defined by
\[
  s(x,y) = 
  \inf \bigl\{\max\{k, n\} \colon k, n \geq 0, \, T^k x = T^n y \bigr\}.
\]

\begin{lemma}
  \label{lemma:abb}
  Let \(n \geq 0\) and \(\rho\) be a probability measure on
  \(M\) such that \(T_*^n \rho\) is regular.
  Then there exists a measure \(\hrho\) on \(M \times M\)
  with marginals \(\rho\) and \(\mu\) on the first and second
  coordinates respectively, such that \(s(x,y) < \infty\)
  for \(\hrho\)-almost every \((x, y) \in M \times M\).
  
  If there exist \(C_\tau > 0\) and \(\beta > 1\) such that
  \(m(\tau \geq \ell) \leq C_\tau \ell^{-\beta}\) for \(\ell \geq 1\),
  then 
  \(
    \hrho (s \geq \ell) \leq C (\ell-n)^{-\beta}
  \)
  for \(\ell \geq n+1\) and some constant \(C>0\).
  
  If there exist constants \(C_\tau > 0\), \(\alpha > 0\) and
  \(\gamma \in (0,1]\) such that
  \(m(\tau \geq \ell) \leq C_\tau e^{-\alpha \ell^{\gamma}}\)
  for \(\ell \geq 1\), then
  \(
    \hrho (s \geq \ell) \leq C e^{-A (\ell-n)^\gamma}
  \)
  for \(\ell \geq n + 1\) and some constants \(A \in (0, \alpha)\)
  and \(C > 0\).

  In both cases above, the constants \(C\) and \(A\) depend continuously 
  (only) on \(R\), \(\xi\), \(C_\tau\), \(\beta\),
  \(\alpha\) and \(\gamma\).
\end{lemma}

\begin{proof}
  Lemma~\ref{lemma:b87e} provides us with the
  decomposition 
  \(\rho = \sum_{a \in \cA} \bP(a) \rho_a\)
  such that
  \(T_*^{n+t(a)} \rho_a = \mu\) for every \(a\).
  
  For \(k \geq 0\) define \(U_k \colon M \to M \times M\),
  \(U_k(x) = (x, T^k x)\).
  Define
  \[
    \hrho = \sum_{a \in \cA} \bP(a) \, (U_{n+t(a)})_* \rho_a.
  \]
  It is clear that the marginals of \((U_{n+t(a)})_* \rho_a\)
  on the first and second components are \(\rho_a\) and \(\mu\)
  respectively. Therefore the marginals of \(\hrho\) are
  \(\rho\) and \(\mu\).

  Observe that \(s(x,y) \leq n+ t(a)\) for 
  \((U_{n+t(a)})_* \rho_a\)-almost every \((x,y) \in M \times M\).
  Thus \(s < \infty\) for \(\hrho\)-almost every \((x,y) \in M \times M\).
  
  It remains to estimate \(\hrho(s \geq \ell)\). Note that
  \(\hrho(s \geq \ell) \leq \bP (t \geq \ell - n)\).
  The results follow directly from
  Propositions~\ref{prop:fwqjq} and~\ref{prop:afsame}.
\end{proof}

\subsection{Proof of Theorem~\ref{thm:yeop}}
\label{proof:thm:yeop}

  By Lemma~\ref{lemma:abb}, for every \(z \in E\) there exists
  a probability measure \(\hrho_z\) on \(M \times M\) with marginals
  \(\rho_z\) and \(\mu\) respectively such that \(s < \infty\)
  almost surely.

  \begin{remark}
    \label{rmk:mf}
    In Proposition~\ref{prop:vaoe} and Lemma~\ref{lemma:b87e},
    we construct the measures \(\rho_a\), \(a \in \cA\)
    (as in Lemma~\ref{lemma:b87e}) by explicit formulas,
    and it is a straightforward verification that,
    as long as \(\rho_z\) is a measurable family, so are the respective
    \(\rho_{z,a}\) for each \(a \in \cA\). Further, \(\hrho_z\) are explicitly
    constructed from \(\rho_{z,a}\) in Lemma~\ref{lemma:abb}, so the
    family \(\hrho_z\) is measurable.
  \end{remark}
  
  Define
  \(
    \hrho = \int_E \hrho_z \, d\varkappa(z)
  \).
  Then the marginals of \(\hrho\) are
  \(\rho\) and \(\mu\) respectively, and \(s < \infty\) 
  almost surely with respect to \(\hrho\).

  It remains to estimate the tails \(\hrho(s \geq n)\).
  We prove the weak polynomial case, the others are similar.
  Using Lemma~\ref{lemma:abb}, write
  \begin{align*}
    \hrho(s \geq n)
    &= \int_E \hrho_z (s \geq n) \, d\varkappa(z)
    \ll \int_E \min \{1, (n - r(z))^{-\beta}\} \, d\varkappa(z)
    \\ & \leq \varkappa(r \geq n/2) + \int_E (n/2)^{-\beta} \, d\varkappa(z)
    \ll n^{-\beta}
    .
  \end{align*}

\subsection{Proof of Theorem~\ref{thm:yeoc}}
  Assume without loss that \(|v|_\infty \leq 1/2\).
  
  Let \(U_n = (v_n, \mu)\). It follows from Theorem~\ref{thm:yeop}
  that the processes \(\{X_n, n \geq 0\}\) and \(\{U_n, n \geq 0\}\) 
  can be redefined on the probability space
  \((M \times M, \hrho_{XU})\) where \(s<\infty\) \(\hrho_{XU}\)-almost surely.
  By Remark~\ref{rmk:aaggg}, \(Z_{XU} = \sup_{n}|X_n - U_n| \leq s\),
  thus \(Z_{XU}\) is also finite \(\hrho_{XU}\)-almost surely.
  
  Similarly, \(\{Y_n, n \geq 0\}\) and \(\{U_n, n \geq 0\}\) 
  can be redefined on \((M \times M, \hrho_{YU})\) 
  with \(\hrho_{YU}\)-almost surely finite \(Z_{YU} = \sup_{n}|Y_n - U_n|\).
  
  By Lemma~\ref{lemma:joc}, all three processes 
  \(\{X_n, n \geq 0\}\), \(\{Y_n, n \geq 0\}\) and 
  \(\{U_n, n \geq 0\}\) can be redefined on the same probability space 
  \((\Omega, \bP)\) so that the joint distributions of pairs 
  \(\{(X_n, U_n), n \geq 0\}\) and \(\{(Y_n, U_n), n \geq 0\}\) are as above. 
  Further we work on this probability space.
  
  Observe that \(Z = \sup_{n} |X_n -Y_n| \leq Z_{XU} + Z_{YU}\).
  It follows that \(Z\) is almost surely finite.
  
  It remains to estimate \(\bP(Z \geq x)\) for \(x \geq 0\). The bounds
  follow transparently from Theorem~\ref{thm:yeop} and the relation
  \begin{align*}
    \bP(Z \geq x) 
    & \leq \bP(Z_{XU} \geq x/2) + \bP(Z_{YU} \geq x/2)
    \\ & \leq \hrho_{XU}(s \geq x/2) + \hrho_{YU}(s \geq x/2)
    .
  \end{align*}

\appendix

\section{Joining of couplings}
\label{app:joc}

Suppose that \(X_j\), \(j=1,2,3\) are random variables 
on probability spaces \(\Omega_j\) with values in 
some measurable spaces \(R_j\).

Assume that \(X_1\) and \(X_2\) can be redefined on
a new probability space \(\Omega_{12}\),
so that the joint distribution of \((X_1, X_2)\) has some
useful property, for example that \(|X_1 - X_2| < 1\)
almost surely.

Assume similarly that \(X_2\) and \(X_3\) can be redefined on
a probability space \(\Omega_{23}\) with a joint distribution
of \((X_2,X_3)\) of interest.

Recall that a \emph{Polish space} is a separable completely
metrizable topological space. In this paper we work with
continuous and discrete time random processes, which 
can be viewed as random variables with values in
the space of c\`adl\`ag functions on \([0,\infty)\), 
or \(\bR^\bN\). These spaces are Polish.

Polish spaces are \emph{universally measurable} (see
\cite{S84} for the definition and discussion). This is a technical
but useful property, which allows to \emph{join couplings:}

\begin{lemma}
  \label{lemma:joc}
  If all value spaces \(R_j\) are universally measurable, then
  \(X_1\), \(X_2\) and \(X_3\) can be redefined on the same
  probability space \(\Omega_{123}\), such that
  the distributions of \((X_1,X_2)\) and \((X_2,X_3)\)
  are the same as on \(\Omega_{12}\) and \(\Omega_{23}\)
  respectively.
\end{lemma}

\begin{proof}
  Note that the probability spaces, on which the random variables 
  \(X_j\) are defined, are irrelevant, so we can instead work directly
  with the corresponding probability measures on \(R_j\).
  In this setting the result is proved in \cite[Lemma~7]{S84}.
\end{proof}

\begin{remark}
  It was pointed out by the referee that
  there is an earlier reference \cite[Lemma~A.1]{BP79} for the result
  of Lemma~\ref{lemma:joc} in case when \(R_j\) are separable Banach spaces.
  It is perfectly sufficient for our purposes (c.f.\ \cite[Subsection~3.1]{G10})
  and avoids the concept of universal measurability.
  
  We are happy to mention \cite{BP79}, yet we keep our Lemma~\ref{lemma:joc},
  because it is more general, and may be easier to use.
  For instance, it is not clear how to apply \cite[Lemma~A.1]{BP79}
  for the space of c\'adl\'ag functions with Skorokhod metric:
  it is separable and complete (thus Polish), but without a
  corresponding norm.
\end{remark}

\subsection*{Acknowledgements}
This research was supported in part by a European Advanced Grant {\em StochExtHomog} (ERC AdG 320977).
The author is grateful to Ian Melbourne for support 
and numerous suggestions.
The author is grateful to the anonymous referee for a very thorough review, many helpful comments and
a request to adapt the manuscript for a larger audience.

\end{document}